\documentclass[11pt]{article}

\usepackage[a4paper,margin=2.7cm]{geometry}
\usepackage{amsmath,amssymb,amsthm}
\usepackage{mathtools}
\usepackage{tikz}
\usepackage{booktabs}
\usepackage{hyperref}
\usepackage{xcolor}
\usepackage{enumitem}
\newtheorem{theorem}{Theorem}
\newtheorem{lemma}[theorem]{Lemma}
\newtheorem{conjecture}[theorem]{Conjecture}
\newtheorem{corollary}[theorem]{Corollary}
\newtheorem{problem}[theorem]{Problem}
\newtheorem{question}[theorem]{Question}

\newtheorem{observation}[theorem]{Observation}
\theoremstyle{definition}

\theoremstyle{remark}

\title{ Disproving the Petersen Coloring Conjecture: Theoretical Analysis and an Infinite Family of Counterexamples}

\author{{\sc J.\ Goedgebeur}$^{1,2}$,
{\sc J.\ Jooken}$^{3,1}$, {\sc E.\ M\'a\v cajov\'a}$^{4}$,\\ {\sc  D.\ Mattiolo}$^{1}$, {\sc G.\ Mazzuoccolo}$^{5}$, {\sc N.\ Ulyanov}$^{6}$ \\[3mm]
\small $^{1}$ Department of Computer Science,\\ \small KU Leuven Kulak, 8500 Kortrijk, Belgium.\\[2mm]
\small $^{2}$ Department of Mathematics, Computer Science and Statistics,\\ \small Ghent University,
9000 Ghent, Belgium.\\[2mm]
\small $^{3}$ Mathematical Institute, Leiden University,\\ \small Einsteinweg 55, 2333 CC Leiden, The Netherlands.\\[2mm]
\small $^{4}$ Department of Computer Science, 
Faculty of Mathematics, Physics and Informatics,\\ \small Comenius University, 84248 Bratislava, Slovak Republic
\\[2mm]
\small $^{5}$ Department of Physics, Informatics and Mathematics,\\ \small University of Modena and Reggio Emilia, 41125 Modena, Italy.\\[2mm]
\small $^{6}$ Independent Researcher, Berlin, Germany.
}

\date{}

\begin{document}
\maketitle

\begin{abstract}
\noindent 
In 1988, Jaeger conjectured that every bridgeless cubic graph $G$ admits a Petersen coloring; that is, a map $E(G) \to E(P)$ mapping any two adjacent edges of $G$ to two adjacent edges of the Petersen graph $P$.  A positive resolution of Jaeger’s conjecture would have immediately resolved several other famous and
long-standing problems in graph theory. 
In July 2026, a 68-vertex counterexample was announced on X. Shortly afterwards, Putman independently presented two non-isomorphic 112-vertex counterexamples, relying solely on computer-assisted verification. In this paper, we present two counterexamples of order $52$, currently the smallest known, and provide a purely theoretical proof. In the second part, we construct an infinite family of cyclically $4$-edge-connected cubic graphs without a Petersen coloring for every even order at least $60$. Additionally, through computational verification, we show that any counterexample must have order at least $40$. Moreover, we show that our counterexamples provide a negative answer to other related problems. Finally, we conclude the paper by discussing key open problems and highlighting avenues for future work.

\end{abstract}


\section{Introduction}

A \emph{Petersen coloring} (or $P$-\emph{coloring}) of a cubic graph $G$ is a map $\varphi \colon E(G) \to E(P)$ mapping any two adjacent edges of $G$ to two adjacent edges of the Petersen graph $P$. In 1988, Jaeger formulated the following celebrated conjecture.

\begin{conjecture}[Petersen Coloring Conjecture~\cite{J88}]\label{PCC}
    Every bridgeless cubic graph admits a $P$-coloring.
\end{conjecture}

For nearly four decades, the Petersen Coloring Conjecture has been a central benchmark in structural graph theory. It was intended to provide a unified explanation of the properties of cycle and matching structures in cubic graphs by tracing them back to the Petersen graph. 

An analogous result for graph minors was recently settled in~\cite{Inoue}, ending  a long-standing research effort spanning several decades. Building upon the landmark works of Robertson, Seymour, and Thomas~\cite{RST97, RST17}, along with Edwards et al.~\cite{ESST16}, this result completes the framework arising from Tutte's $3$-edge-coloring conjecture~\cite{Tutte66}, which asserted that any bridgeless cubic graph lacking a Petersen minor is $3$-edge-colorable.

For cubic graphs, admitting a $P$-coloring is equivalent to admitting a \emph{cycle-continuous map} to the Petersen graph, that is, an edge-mapping under which the preimage of every cycle is a cycle (see~\cite{DNR}). Hence, under Conjecture \ref{PCC}, the cycle space of the Petersen graph was expected to be sufficient to describe the cycle space of every bridgeless cubic graph.

%For a long time, researchers believed that the Petersen graph could similarly capture all possible structural variety in the context of cycle-continuous maps, that is, maps between edge-sets such that the preimage of every cycle in the target graph is a cycle in the source graph.

The Petersen coloring conjecture lies at the core of the theories of flows and colorings of graphs, which themselves share strong connections with matroid theory~\cite{J88}, mathematical physics (such as zero-temperature states of spin models, see~\cite{BEMPS, J92}), polyhedral optimization~\cite{Edm, F71, Sch}, and topology~\cite{HN, Kau}. 

A positive answer to the Petersen coloring conjecture would have immediately solved several other famous and long-standing problems in these areas. Most notably, it is well known~\cite{J88} that it implies both the \emph{Berge--Fulkerson Conjecture}~\cite{F71} (which asserts that every bridgeless cubic graph contains six perfect matchings covering every edge exactly twice) and the \emph{5-Cycle Double Cover Conjecture}~\cite{C85,P82} (stating that every bridgeless graph admits a cycle double cover consisting of at most five even subgraphs). As a result, the conjecture has served as a cornerstone for understanding the edge-structure and cycle spaces of snarks and regular graphs.

However, this paradigm turned out to be false, as the recent discovery of counterexamples demonstrates. The refutation of the conjecture opens up an entirely new landscape, revealing a rich variety of cubic graphs whose cycle structure exhibits a level of complexity exceeding that of the Petersen graph.

A counterexample on $68$ vertices to Conjecture~\ref{PCC}, credited to GPT-5.6 Sol Ultra, was posted on X in July 2026 by the account \texttt{@NeuralReformist}\footnote{\url{https://x.com/NeuralReformist/status/2080369979388805555}}. Moreover, in August 2026, with the use of an AI model, Putman~\cite{P26,P26b} provided two non-isomorphic counterexamples on $112$ vertices. However, these verifications remain strictly computational, and the mere exhibition of such counterexamples provides no structural insight into why these graphs fail to admit the required mapping.

By studying these examples, we identified why they fail to admit a Petersen coloring. This insight allowed us to construct counterexamples on just $52$ vertices, for which we provide a complete, purely theoretical proof in Section~\ref{section:52}. 

Crucially, uncovering these underlying reasons allows us to gain a deeper understanding of the conjecture's failure, opening directions for future research to investigate which of its famous consequences, such as the Berge-Fulkerson Conjecture or the 5-Cycle Double Cover Conjecture, might likewise be false or, conversely, remain valid.

The counterexamples presented above are cyclically $4$-edge-connected. Recall that a cubic graph is \emph{cyclically $k$-edge-connected} if no set of fewer than $k$ edges disconnects it into two components, each containing a cycle. While applying the standard $3$-sum operation to a $3$-edge-connected counterexample easily yields an infinite family of such graphs, a natural question is whether there exist infinitely many cyclically $4$-edge-connected graphs without a Petersen coloring. In Section~\ref{sec:infinite}, we construct such a family for every even order of at least $60$, leaving the existence of a cyclically $5$-edge-connected counterexample as an open problem.
%The above-mentioned counterexamples are cyclically $4$-edge-connected, where a cubic graph is \emph{cyclically $k$-edge-connected} if no set of fewer than $k$ edges disconnects it into two components, each containing a cycle. Starting from an arbitrary $3$-edge-connected counterexample, the existence of an infinite family of $3$-edge-connected counterexamples follows directly by applying the standard $3$-sum operation. Therefore, a natural question is whether there exist infinitely many cyclically $4$-edge-connected cubic graphs without a Petersen coloring. 
%In Section~\ref{sec:infinite}, we construct a family of cyclically $4$-edge-connected cubic graphs without a Petersen coloring for every even order of at least $60$, while leaving the existence of a cyclically $5$-edge-connected counterexample as an open problem. 

Finally, in Section~\ref{sec:openProblems}, we show that our counterexamples answer related open problems in the negative. Furthermore, by extending previous exhaustive searches, which had verified Conjecture~\ref{PCC} up to $36$ vertices, to all candidate graphs on $38$ vertices, we establish that any counterexample must have order at least $40$. This narrows the gap for the order of a minimum counterexample to between $40$ and $52$.

\section{Two relevant multipoles}

A \emph{multipole} $M$ is a structure consisting of a vertex set $V(M)$ and an edge set $E(M)$, where every edge has two ends, each of which is either incident to a vertex in $V(M)$ or left unattached. An end of an edge that is not incident to any vertex is called a \emph{free end}. An edge with exactly one free end and one end incident to a vertex is referred to as a \emph{dangling edge}. A multipole possessing exactly $k$ free ends is called a \emph{$k$-pole}. The set of dangling edges in a multipole, say $d_1,d_2,\ldots, d_k$, is endowed with a linear order which can be specified as $M(d_1,d_2,\ldots, d_k)$. If the order of the dangling edges is not important, we simply write $M$. A $k$-pole is said to be \emph{cubic} if every vertex in $V(M)$ has degree exactly $3$ (counting both edges and dangling edges incident to it).

A $P$-coloring of a cubic multipole $M$ is a map $\varphi\colon E(M)\to E(P)$ such that for every vertex $v$ of $M$, there is a vertex $w_v$ of $P$ satisfying $\varphi(\partial_M(v))=\partial_P(w_v)$, where $\partial_G(u)$ denotes the set of edges and dangling edges incident to the vertex $u$ in $G$.
In the following, we construct cubic graphs by combining specific multipoles while controlling the restrictions that any Petersen coloring induces on their dangling edges.

For a $k$-pole $M(d_1,\ldots,d_k)$, we define its \emph{$P$-coloring set} as 
\[
P\text{-Col}(M) = \{(\varphi(d_1),\varphi(d_2), \ldots,\varphi(d_k)) \colon \varphi \text{ is a Petersen coloring of } M\}.
\]

To combine multipoles, we use the following standard operation: given two dangling edges $d_1$ and $d_2$ belonging to disjoint multipoles and incident to vertices $u$ and $v$ respectively, to \emph{join} $d_1$ and $d_2$ means to remove both dangling edges and replace them with a single edge $uv$.

\subsection{The 4-pole \textit{W}}

Let $P$ denote the Petersen graph. Throughout the paper, we consider the labeling of the edges of $P$ as in Figure~\ref{fig:petersen_labelled}. We define the $4$-pole $W(i_1,i_2,o_1,o_2)$ (commonly referred to as the \emph{Petersen $4$-pole}) as the multipole obtained from $P$ by removing two adjacent vertices, say $u$ and $v$, along with the edge $uv$ connecting them. The resulting structure $W$ consists of $8$ vertices and $4$ dangling edges, which correspond to the edges of $P$ that were incident to $u$ and $v$ (excluding $uv$), see Figure~\ref{fig:petersen_labelled}. Since every vertex in $W$ retains degree $3$, $W$ is a cubic $4$-pole.
We denote the two dangling edges originally incident to $u$ by $i_1$ and $o_1$, and the two dangling edges originally incident to $v$ by $i_2$ and $o_2$.

\begin{figure}[ht]
\centering
\includegraphics[width=8cm]{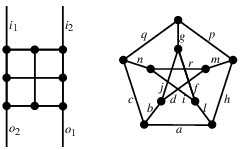}
\caption{Labeling of the edges of $W$ and $P$ as used in the proof of Lemma~\ref{lem:W_coloring} and Lemma~\ref{lem:5pole_configs}.}
\label{fig:petersen_labelled}
\end{figure}

%\gi{this labeling of the Petersen is quite crazy, we should change it for the journal version.} \ed{this labeling makes sense to me, I would keep it}

\begin{figure}
    \centering
    \includegraphics[width=0.7\linewidth]{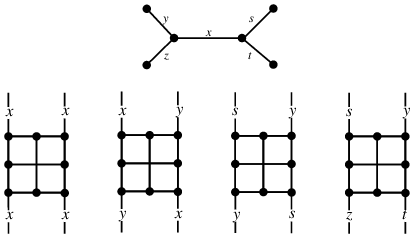}
    \caption{The four possibilities for a Petersen coloring of dangling edges of $W$.}
    \label{fig:W}
\end{figure}

Let $C$ be a cubic $5$-pole consisting of a $5$-cycle $v_1v_2v_3v_4v_5$ with five dangling edges $e_1,\dots,e_5$. For each $i$, $e_i$ is the dangling edge incident to $v_i$.

\begin{lemma}\label{lem:5-cycle}

Let $f$ be a $P$-coloring of $C$.
Then, up to automorphisms of $P$ and $C$, $$(f(e_1),\dots,f(e_5)) \in \{(a,a,a,b,c),(h,a,l,b,b),(m,g,n,b,l)\}.$$
    
\end{lemma}

\begin{proof}
    The $P$-coloring $f$ induces a natural vertex-map $f_V\colon V(C)\to V(P)$ as follows: $f_V(v)=w$ if $w$ is the vertex of $P$ such that $f(\partial_C(v))=\partial_P(w)$. 

    We first prove the following property of $f_V$. Indices are taken modulo $5$.
    \begin{equation}\label{property f_V}
        \text{For every } i\in \{1,\dots,5\}, \text{ if } f_V(v_i)= f_V(v_{i+2})= x\in V(P), \text{ then } f_V(v_{i+1})=x.
    \end{equation} 
    Assume, for a contradiction, that $f_V(v_{i+1})\ne x$. Then, it must be adjacent to both $f_V(v_i)$ and $f_V(v_{i+2})$. Since $f(v_iv_{i+1})\ne f(v_{i+1}v_{i+2})$ and $P$ does not contain parallel edges we obtain a contradiction. We proceed considering Im$(f_V)$.

    If Im$(f_V)$ consists of five vertices, then the only possibility is that the graph induced by Im$(f_V)$ in $P$ is a $5$-cycle. Then, up to automorphisms of $P$ and $C$, we deduce that $(f(e_1),\dots,f(e_5)) = (m,g,n,b,l)$. This case is depicted on the right-hand side of Figure \ref{fig:5cycle}.

Then $|\text{Im}(f_V)| < 5$ and at least two vertices of $C$ are mapped to the same vertex $x \in V(P)$. We distinguish two cases based on their distance in $C$.
Suppose first that two non-consecutive vertices in $C$ are mapped to $x$. By \eqref{property f_V}, their common neighbor in $C$ must also be mapped to $x$. Let $u$ and $v$ be the remaining two vertices of $C$. If $f_V(u) \ne f_V(v)$, we reach a contradiction: indeed, if neither is mapped to $x$, the images $x, f_V(u),$ and $f_V(v)$ would induce a $3$-cycle in $P$, which is impossible; if one of them is mapped to $x$, we would obtain a sequence of images $x,y,x$ (with $y \ne x$) along $C$, directly violating \eqref{property f_V}. Therefore, $u$ and $v$ must share the same image. If this image is $x$, all five vertices of $C$ map to $x$, yielding the case $(f(e_1),\dots,f(e_5)) = (a,a,a,b,c)$. If this image is some $y \ne x$, we obtain the case $(f(e_1),\dots,f(e_5)) = (h,a,l,b,b)$.
Now we can assume that no two vertices at distance $2$ in $C$ share the same image. Because $|\text{Im}(f_V)| < 5$, there must be two adjacent vertices in $C$ mapped to the same vertex $x \in V(P)$. None of the remaining three vertices can be mapped to $x$, as that would force two vertices at distance $2$ to share the image $x$.
If these three remaining vertices map to three pairwise distinct vertices in $P$, their images, together with $x$, would form a $4$-cycle in $P$, contradicting the fact that the Petersen graph has girth $5$.
If instead they map to exactly two distinct vertices $y,z \in V(P)$, then one image, say $y$, is assigned to two vertices of $C$. To avoid sharing an image at distance $2$, these two vertices must be adjacent in $C$. However, this forces the images $x, y,$ and $z$ to form a $3$-cycle in $P$, which is once again a contradiction.

%    Then, we assume that $|$Im$(f_V)|<5$. Since $P$ has girth $5$, Im$(f_V)$ induces a path $J$ in $P$. 
    
%  Since the $v_j$s induce a cycle in $C$, $|f_V^{-1}(u)|\ge2$ for every $u\in V(J).$ Therefore, since $|V(C)|=5$, we deduce that $|V(J)|\le2$.

%If $|V(J)|=1$, then $f_V$ maps all vertices of $C$ to a single vertex of $P$. Up to automorphisms of $P$ and $C$, this yields $(f(e_1),\dots,f(e_5)) = (a,a,a,b,c)$, which corresponds to the case depicted on the left-hand side of Figure~\ref{fig:5cycle}. If $|V(J)|=2$, let $V(J)=\{u,w\}$. Thus, $f_V^{-1}(u)$ and $f_V^{-1}(w)$ have cardinality $3$ and $2$, respectively. By \eqref{property f_V}, the preimages cannot alternate along $C$, then up to automorphisms of $P$ and $C$, this uniquely determines $(f(e_1),\dots,f(e_5)) = (h,a,l,b,b)$, which is shown in the center of Figure~\ref{fig:5cycle}.
\end{proof}

Now, we establish the following lemma about the $P$-coloring set of $W$.

\begin{lemma} \label{lem:W_coloring}

For any Petersen coloring of $W$, there exists an edge $x \in E(P)$, with $\{y,z\}$ and $\{s,t\}$ denoting the pairs of edges incident to its respective endpoints, such that the sequence of colors assigned to $(i_1, i_2, o_1, o_2)$ is exactly one of the following:
\begin{itemize}
    \item $(x,x,x,x)$;
    \item $(x,y,x,y)$;
    \item $(s,y,s,y)$;
    \item $(s,y,t,z)$.
\end{itemize}
\end{lemma}
\begin{proof} 
 Fix one 5-cycle, denoted by $C$, in $W$. In any Petersen coloring of $W$, the restriction of the coloring to $C$ and its five incident edges must constitute a valid Petersen coloring of a 5-pole formed by a 5-cycle. Up to automorphisms of the Petersen graph and the $5$-cycle, there are exactly three different ways to color such a 5-pole by Lemma \ref{lem:5-cycle}.

\begin{figure}
    \centering
    \includegraphics[width=0.9\linewidth]{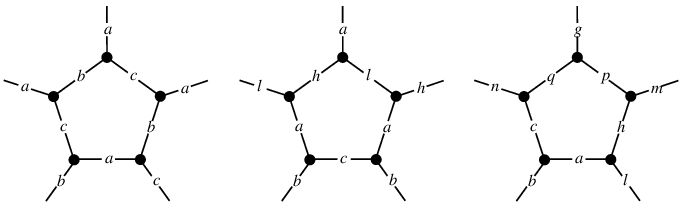}
    \caption{The only three possible Petersen colorings of a $5$-cycle $C$, up to symmetries of $P$ and $C$. }
    \label{fig:5cycle}
\end{figure}

    To find all possible colorings of $W$, we systematically apply these three colorings to the fixed 5-cycle $C$ in all possible ways up to symmetries and test whether the partial coloring can be extended to the remaining edges and the four dangling edges of $W$.

    \begin{itemize}
        \item \textbf{Case 1:} Among all possible rotations of the left coloring in Figure~\ref{fig:5cycle} of $C$, those that successfully extend to the remaining edges yield exactly the sequences $(x, x, x, x)$ and $(x, y, x, y)$ for the dangling edges.
                
        \item \textbf{Case 2:} When testing the middle coloring on $C$, propagating the forced colors to the remaining vertices of $W$ leads to a conflict in two possible alignments. In these non-extendable alignments, local constraints force two adjacent edges in the uncolored part of $W$ to receive the same color or receive colors that are not adjacent in $P$. However, in exactly one case, the coloring successfully extends, uniquely producing the sequence $(s, y, s, y)$.
        
        \item \textbf{Case 3:} Finally, we consider the right coloring in Figure~\ref{fig:5cycle}, which corresponds to an injective mapping of the 5-cycle into $P$. The extension to the remaining vertices of $W$ is uniquely determined by the adjacencies in $P$. This configuration successfully extends to $W$ and produces the remaining sequence for the dangling edges: $(s, y, t, z)$.
    \end{itemize}

     Therefore, the sequence of colors assigned to $(i_1, i_2, o_1, o_2)$ is exactly one of the four configurations listed in the statement.
\end{proof}

Before stating the next result, we formally recall the concept of distance between edges. For any two edges $e, f \in E(P)$, the distance $\text{dist}_P(e, f)$ is defined as the shortest path distance between their corresponding vertices in the line graph of $P$.

\begin{corollary}\label{cor:Wdistances}
Let e and f be two dangling edges of $W$ and let $\varphi$ be a Petersen coloring of $W$. Then the distance in the Petersen graph $P$ between the colors assigned to $e$ and $f$ satisfies $\text{dist}_P(\varphi(e), \varphi(f)) \le 2$.\end{corollary}

\subsection{The 5-pole \textit{F}}\label{sec:5pole}

 Starting from the $4$-pole $W$ introduced in the previous section, we construct the $5$-pole $F$ which will play a crucial role in the construction of a counterexample to Conjecture~\ref{PCC}.

\begin{figure}[ht]
\centering
\includegraphics[width=10cm]{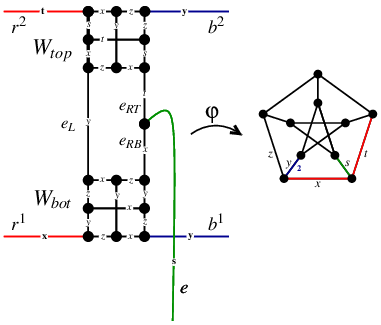}
\caption{The $5$-pole $F$ and an example of a Petersen coloring and the corresponding configuration in the Petersen graph.}
\label{fig:config_example}
\end{figure}

We construct $F$ by joining two copies of $W$, denoted $W_{\text{top}}$ and $W_{\text{bot}}$. We refer to  $i_1, i_2$ as the \emph{top} dangling edges of $W$ and $o_1, o_2$ as its \emph{bottom} dangling edges; for $W_{\text{top}}$ and $W_{\text{bot}}$, we denote these by $i_{1t}, i_{2t}, o_{1t}, o_{2t}$ and $i_{1b}, i_{2b}, o_{1b}, o_{2b}$, respectively. 
We connect $o_{2t}$ to $i_{1b}$ to form an edge $e_L$, and join $o_{1t}$ to $i_{2b}$, subdividing this second connection with a new vertex to attach a new dangling edge $e$ (yielding the edges $e_{\text{RT}}$ and $e_{\text{RB}}$).
Two dangling edges of $W_{\text{top}}$ and two dangling edges of $W_{\text{bot}}$ that remain unjoined are identified as dangling edges of $F$. Together with $e$, they form the five dangling edges of $F$ and are renamed according to their position and color as follows (see also Figure~\ref{fig:config_example}):

\begin{itemize} 
\item The dangling edges $i_{1t}$ and $o_{2b}$  become dangling edges of $F$ called $r^1$ and $r^2$, respectively, and we set $R = \{r^1, r^2\}$; these two edges will be called \emph{red}; 
\item The dangling edges $i_{2t}$ and $o_{1b}$ become dangling edges of $F$ called $b^1$ and $b^2$, respectively, and we set $B = \{b^1, b^2\}$; these two edges will be called \emph{blue}; 
\item The newly attached dangling edge $e$ will be called \emph{green}. 
\end{itemize}

Let $\varphi$ be a Petersen coloring of the $5$-pole $F$. We define the \emph{configuration of $\varphi$} as the ordered partition of the multiset of edges $\{\varphi(r^1), \varphi(r^2), \varphi(b^1), \varphi(b^2), \varphi(e)\}$ into the tuple of multisets $(\varphi(R), \varphi(B), \{\varphi(e)\})$. The edges in these three multisets are called red, blue, and green, respectively.

It is worth highlighting the formal distinction between the configuration of a coloring $\varphi$ and the $P$-coloring set $P\text{-Col}(F)$. While $P\text{-Col}(F)$ is defined as a set of strictly ordered tuples $(\varphi(r^1), \varphi(r^2), \varphi(b^1), \varphi(b^2), \varphi(e))$, a configuration abstracts away certain structural symmetries by grouping these entries into multisets. Specifically, the set of all configurations can be formally understood as a quotient of $P\text{-Col}(F)$ under an equivalence relation, where two tuples yield the same configuration if they are the same up to switching the first two entries (permuting the red edges $r^1$ and $r^2$) or switching the third and fourth entries (permuting the blue edges $b^1$ and $b^2$).

 Consider the Petersen coloring $\varphi$ of $F$ in Figure~\ref{fig:config_example}. Following the notation of the figure, we have that the configuration of $\varphi$ is the tuple of multisets $(\{x,t\},\{y,y\},\{s\})$.

 We say that two configurations are equivalent if there exists an automorphism of the Petersen graph that maps one to the other.

In what follows, the nine configurations $\mathcal{C}_1,\mathcal{C}_2,\ldots,\mathcal{C}_9$ depicted in Figure~\ref{fig:configurations} will play an important role.

For example, the configuration $(\{x,t\},\{y,y\},\{s\})$ in Figure~\ref{fig:config_example} is equivalent to Configuration 2 in the list in Figure~\ref{fig:configurations}.

\begin{figure}[htbp]
\centering
\includegraphics[width=8cm]{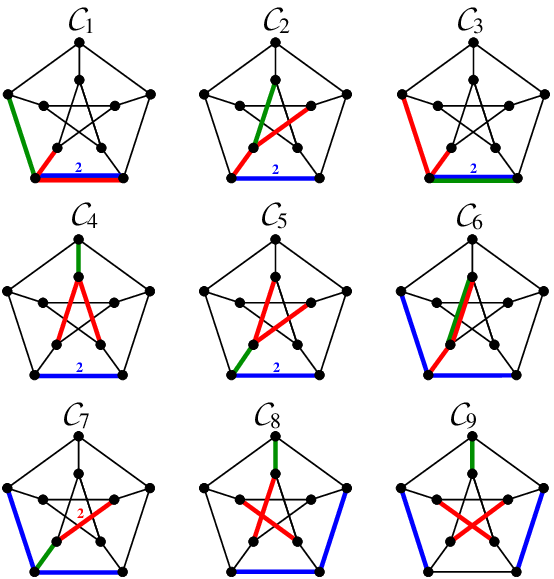}
\caption{The nine possible configurations of a Petersen coloring of the $5$-pole $F$.}
\label{fig:configurations}
\end{figure}

\begin{lemma}\label{lem:5pole_configs}
Up to equivalence, the only possible configurations that can occur in the coloring of $F$ are $\mathcal{C}_1,\mathcal{C}_2,\ldots,\mathcal{C}_9$.
\end{lemma}

\begin{proof}
Let $\varphi$ be a Petersen coloring of the 5-pole $F$. The 5-pole $F$ is constructed by joining two copies of the 4-pole $W$, as explained at the beginning of Section~\ref{sec:5pole}.

%In the rest of this proof, we consider the labeling of the edges of the Petersen graph as in Figure~\ref{fig:petersen_labelled}.

Without loss of generality, due to the edge-transitivity of the Petersen graph $P$, we may assume the edge $e_L$ is colored $\varphi(e_L) = a \in E(P)$. 
By Corollary~\ref{cor:Wdistances}, any valid pair of dangling edges on one side of $W$ must be at a distance of 0, 1, or 2 in $P$. Since we fix $\varphi(e_L) = a$, this requires $\text{dist}_{P}(\varphi(e_{RT}), a) \in \{0,1,2\}$ and $\text{dist}_{P}(\varphi(e_{RB}), a) \in \{0,1,2\}$. 
Moreover, since $e_{RT}$ and $e_{RB}$ are adjacent in $F$,  their images in $P$ must also be adjacent. 

Up to the automorphisms of $P$ that fix the edge $a$ and by the horizontal symmetry of $F$, there are exactly five equivalence classes for the ordered pair of adjacent edges $(\varphi(e_{RT}), \varphi(e_{RB}))$ bounded by distance 2 from $a$:  $(a, b)$, $(b, c)$, $(d, j)$, $(b, d)$, and $(j,f)$. From these cases, using Lemma~\ref{lem:W_coloring}, we will deduce all possible configurations of $\varphi$.

\begin{itemize}
    \item \textbf{Case 1: $(\varphi(e_{RT}), \varphi(e_{RB})) = (a, b)$} \\
     It follows $\varphi(e)=c$. Moreover,  since the two bottom dangling edges of $W_{top}$ have colors $(a, a)$, this uniquely forces its top dangling edges to $(a, a)$. The top dangling edges of $W_{bot}$ have $(a, b)$, forcing its bottom dangling edges to $(b, a)$. This yields exactly one configuration, denoted $\mathcal{C}_1$.
    
    \item \textbf{Case 2: $(\varphi(e_{RT}), \varphi(e_{RB})) = (b, c)$} \\
    It follows $\varphi(e)=a$. Moreover, the bottom dangling edges of $W_{top}$ have $(a, b)$, uniquely forcing the top dangling edges to $(b, a)$. The top dangling edges  of $W_{bot}$ have $(a, c)$, uniquely forcing the bottom dangling edges to $(c, a)$. This yields exactly one configuration, $\mathcal{C}_3$.
    
    \item \textbf{Case 3: $(\varphi(e_{RT}), \varphi(e_{RB})) = (d, j)$} \\
    It follows $\varphi(e)=b$. Moreover, the bottom dangling edges of $W_{top}$ have $(a, d)$, which branches into two valid assignments for the top dangling edges: $(d, a)$ and $(j, c)$. The top dangling edges of $W_{bot}$ have $(a, j)$, branching into $(j, a)$ and $(d, c)$ for the bottom dangling edges. Pairing these branches gives $4$ possibilities. Under the automorphism group of $P$, these collapse into two distinct configurations, each appearing twice. We denote these $\mathcal{C}_5$ and $\mathcal{C}_7$.

    \item \textbf{Case 4: $(\varphi(e_{RT}), \varphi(e_{RB})) = (b, d)$} \\
    It follows $\varphi(e)=j$. Moreover, the bottom dangling edges of $W_{top}$ have $(a,b)$ forcing the top edges to $(b, a)$. The top dangling edges of $W_{bot}$ have $(a, d)$, branching into $(d, a)$ and $(j, c)$. This gives $2$ possibilities, mapping to distinct configurations $\mathcal{C}_2$ and $\mathcal{C}_6$.
    
    \item \textbf{Case 5: $(\varphi(e_{RT}), \varphi(e_{RB})) = (j, f)$} \\
    It follows $\varphi(e)=g$. Moreover, the bottom dangling edges of $W_{top}$ have $(a, j)$, branching into $(j, a)$ and $(d, c)$. The top dangling edges of $W_{bot}$ have $(a, f)$, branching into $(f, a)$ and $(i, h)$. This generates $4$ possibilities. Under automorphisms of $P$, these fall into three distinct configurations: $\mathcal{C}_8$ (appearing twice), $\mathcal{C}_4$, and $\mathcal{C}_9$.\qedhere
\end{itemize}
\end{proof}

\section{Two counterexamples of order 52} \label{section:52}

We consider the cubic graph $G$ of order 52 depicted in Figure~\ref{fig:52graph}. The graph $G$ is constructed from three copies of the 5-pole $F$, say $F_1, F_2, F_3$, connected cyclically and with a single central vertex $v$, as shown in the schematic representation in Figure~\ref{fig:graph_structure}.

\begin{figure}[htbp]
\centering
\includegraphics[width=6cm]{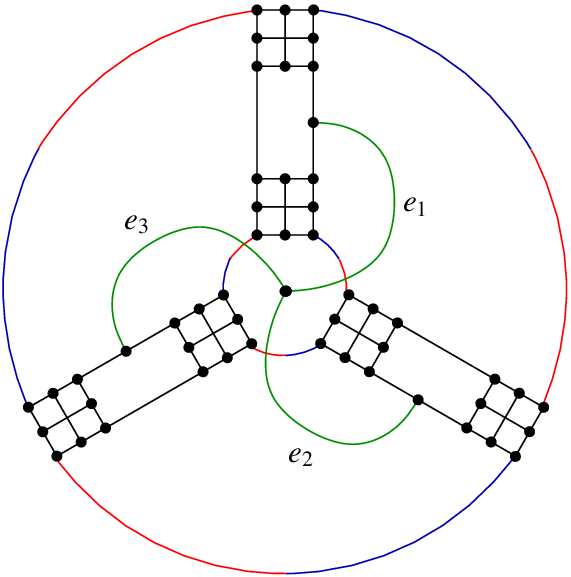}

\caption{A smallest known cubic graph with no Petersen coloring.}
\label{fig:52graph}
\end{figure}

The connections in $G$ are defined cyclically modulo 3, where the edges $B_i$ of $F_i$ are joined to the edges $R_{i+1}$ of $F_{i+1}$, and the three middle edges $e_1, e_2, e_3$ meet at the central vertex $v$.

\begin{theorem}
The cubic graph $G$ of order 52 does not admit a Petersen coloring.
\end{theorem}

\begin{proof}
Suppose, for the sake of contradiction, that $G$ admits a Petersen coloring $\varphi: E(G) \to E(P)$. Up to automorphisms of $P$, the restriction of $\varphi$ to any of the three 5-poles $F_i$ produces exactly one of the configurations $\mathcal{C}_1,\mathcal{C}_2,\ldots,\mathcal{C}_9$.

The Petersen coloring $\varphi$ satisfies the following conditions for all $i \in \{1,2,3\}$ (indices modulo 3):
\begin{enumerate}[label=T\arabic*)]
    \item $\varphi(B_i) = \varphi(R_{i+1})$;
    \item $\{\varphi(e_1), \varphi(e_2), \varphi(e_3)\}$ is a set of three distinct edges in $P$ sharing a common vertex.
\end{enumerate}

\begin{figure}[htbp]
\centering
\includegraphics[width=4cm]{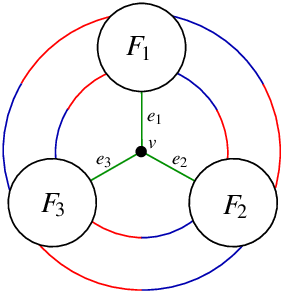}
\caption{Schematic representation of the cubic graph $G$ composed of three 17-vertex 5-poles $F_1, F_2, F_3$ joined cyclically around a central vertex $v$ via middle edges $e_1, e_2, e_3$.}
\label{fig:graph_structure}
\end{figure}

We now analyze the possible transitions of configurations  between consecutive 5-poles $F_i$ and $F_{i+1}$ under these two constraints.

\begin{itemize}
    \item \textbf{Elimination of $\mathcal{C}_8$ :}
 First, we prove that if $F_i$ has configuration $\mathcal{C}_8$, then there is no possible configuration for $F_{i+1}$.
 Indeed, in configuration $\mathcal{C}_8$, the multiset $\varphi(B_i)$ consists of two distinct edges in $P$ that share a common vertex $x$. The green edge $\varphi(e_i)$ has both ends at distance 2 from $x$.
   According to T1, the red edges of $F_{i+1}$ must satisfy $\varphi(R_{i+1}) = \varphi(B_i)$, which means that $\varphi(R_{i+1})$ must consist of two distinct edges incident to $x$. In all configurations with this property, i.e. $\mathcal{C}_1, \ldots, \mathcal{C}_6$, the green edge $\varphi(e_{i+1})$ is also incident to $x$. Hence $\varphi(e_i)$ and $\varphi(e_{i+1})$ are not adjacent, which contradicts T2.

    \item \textbf{Elimination of $\mathcal{C}_9$:} In the configuration $\mathcal{C}_9$, the multiset $\varphi(B_i)$ consists of two non-adjacent edges in $P$. By T1, the multiset $\varphi(R_{i+1})$ must also consist of two non-adjacent edges, which restricts the candidate configurations for $F_{i+1}$ exclusively to $\mathcal{C}_8$, which is already excluded, and $\mathcal{C}_9$. Now, we show that a configuration $\mathcal{C}_9$  cannot be followed by another configuration $\mathcal{C}_9$. Indeed, $\mathcal{C}_9$ consists of five distinct edges that form a perfect matching of $P$. By T1, the perfect matching corresponding to the configuration of $F_i$ should share two edges with the perfect matching corresponding to the configuration of $F_{i+1}$, since every two distinct perfect matchings of $P$ share exactly one edge, they are the same set of edges. Consequently, both green edges $\varphi(e_i)$ and $\varphi(e_{i+1})$ belong to the same perfect matching, then they are either the same edge or not adjacent, a contradiction to T2. Consequently, configuration $\mathcal{C}_9$ is also ruled out.

   \item \textbf{Elimination of $\mathcal{C}_7$:} First, we prove that if $F_i$ has configuration $\mathcal{C}_7$, then $F_{i+1}$ must have configuration $\mathcal{C}_6$. Indeed, in configuration $\mathcal{C}_7$, the multiset $\varphi(B_i)$ consists of two distinct edges in $P$ that share a common vertex $x$, to which the green edge $\varphi(e_i)$ is also incident. According to T1, the red edges of $F_{i+1}$ must satisfy $\varphi(R_{i+1}) = \varphi(B_i)$, which means that $\varphi(R_{i+1})$ must consist of two distinct edges incident to $x$. This requirement restricts the candidate configurations for $F_{i+1}$ to $\mathcal{C}_1, \mathcal{C}_2, \mathcal{C}_3, \mathcal{C}_4, \mathcal{C}_5,$ and $\mathcal{C}_6$. Furthermore, in all remaining cases, except $\mathcal{C}_6$, the position of the green edge forces $\varphi(e_i)=\varphi(e_{i+1})$, which is a contradiction to T2.

    Assume now that $F_i$ has configuration $\mathcal{C}_7$, which forces $F_{i+1}$ to have configuration $\mathcal{C}_6$. We show that no valid configuration remains for the third pole $F_{i+2}$.

 Since the configuration $\mathcal{C}_7$ on $F_i$ requires $\varphi(R_i)$ to be a double edge, the multiset $\varphi(B_{i-1})=\varphi(B_{i+2})=\varphi(R_i)$ must have a double edge.
 Moreover, the configuration $\mathcal{C}_6$ of $F_{i+1}$ has $\varphi(B_{i+1})$ consisting of two distinct adjacent edges; then the edges in $\varphi(R_{i+2})$ must also be distinct and adjacent. Finally, the relative position of the edges in $\mathcal{C}_7$ and $\mathcal{C}_6$ forces one of the two edges in $\varphi(R_{i+2})$ to be at distance $3$ from the double edge in $\varphi(B_{i+2})$. Since no configuration has such a property, the choice of $\mathcal{C}_7$ cannot be completed by any configuration on $F_{i+2}$, and $\mathcal{C}_7$ cannot appear as a configuration of $\varphi$.

    \item \textbf{Elimination of $\mathcal{C}_1, \mathcal{C}_2, \mathcal{C}_3, \mathcal{C}_4, \mathcal{C}_5$:} Configurations $\mathcal{C}_1, \mathcal{C}_2, \mathcal{C}_3, \mathcal{C}_4$ and $\mathcal{C}_5$ all require $\varphi(B_i)$ to be a double edge. A double edge in $\varphi(B_i)$ must induce a double edge in $\varphi(R_{i+1})$ in $F_{i+1}$. Since $\mathcal{C}_7$ is the only configuration that generates such a double red edge, by the previous point all such configurations are also impossible.

    \item \textbf{Elimination of $\mathcal{C}_6$ :}
        We prove that configuration $\mathcal{C}_6$ for $F_i$ cannot be followed by another configuration $\mathcal{C}_6$ for $F_{i+1}$.
        Suppose that $F_i$ has configuration $\mathcal{C}_6$, the multiset $\varphi(B_i)$ consists of two distinct edges in $P$ sharing a vertex $x$. The green edge $\varphi(e_i)$ is not adjacent to either of them. By T1, $\varphi(B_i)=\varphi(R_{i+1})$ , which forces $\varphi(e_{i+1})$ to be one of the two edges in $\varphi(B_i)$ and thus not adjacent to $\varphi(e_i)$, contradicting T2.
    \end{itemize}

Since no sequence of configurations is admissible for $F_1, F_2, F_3$, $G$ does not admit a Petersen coloring.
\end{proof}

By switching two suitable edges of the graph considered so far (see Figure~\ref{fig:graph_switched}), we obtain a new example of a cyclically 4-edge-connected cubic graph of order 52 without a Petersen coloring.
The switch only changes the pairing of the two edges joining $B_2$ to $R_3$; therefore the multiset identity $\varphi(B_i)=\varphi(R_{i+1})$
remains valid. Hence, T1 and T2, and consequently the entire proof, are unchanged.

\begin{figure}[htbp]
\centering
\includegraphics[width=4cm]{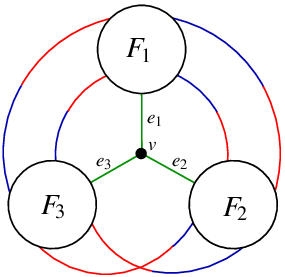}
\caption{Schematic representation of another cubic graph of order $52$ and without a Petersen coloring.}
\label{fig:graph_switched}
\end{figure}

\section{An infinite family of cyclically 4-edge-connected cubic graphs without a \textit{P}-coloring}\label{sec:infinite}

Let $M(e_1,\dots, e_4)$ and $X(a_1,\dots,a_4)$ be two $4$-poles. Moreover, assume that $X$ is a submultipole of a cubic graph $G$: i.e.\ $V(X)\subseteq V(G)$ and $\partial_G(V(X)) = \{a_1,\dots,a_4\}$ is a $4$-edge-cut in $G$. 
Then, the graph induced by $V(G)\setminus V(X)$ plus the dangling edges $a_1,\dots,a_4$ constitutes a $4$-pole $Y$. Consider the graph $H$ obtained from the multipoles $Y$ and $M$ by gluing their dangling edge $a_i$ with $e_i$ for every $i\in \{1,2,3,4\}$. We say that $H$ is obtained by \emph{replacing} $X$ with $M$.

Let $M_1(a_1,a_2,a_3,a_4),M_2(b_1,b_2,b_3,b_4),M_3(c_1,c_2,c_3,c_4),M_4(d_1,d_2,d_3,d_4)$ be the $4$-poles as in Figure~\ref{fig:FandPlus2}. The following lemma was obtained with the help of a computer.

\begin{lemma} \label{lemma:Mka}
  $P\text{-Col}(M_i)=P\text{-Col}(W)$ for every $i\in\{1,2,3,4\}$.
\end{lemma}

\begin{theorem}
  For each even $n\geq60$ there exists a cyclically $4$-edge-connected cubic graph of girth $5$ on $n$ vertices which does not admit a $P$-coloring.
\end{theorem}
\begin{proof}
  We prove by induction that for each even $n\ge60$ there exists a cyclically 4-edge-connected cubic graph of girth 5 and order $n$ which contains the 4-pole $W$ as a subgraph. For the basis we construct graphs on $60,62,64$, and $66$ vertices with the required properties. We start from one of the two graphs on 52 vertices constructed in Section~\ref{section:52}. The graph  contains a copy of $W$, we replace it with one of the 4-poles $M_1,M_2,M_3,M_4$.  Lemma~\ref{lemma:Mka} implies that the constructed graphs $G_{60}, G_{62}, G_{64}$ and $G_{66}$ also do not have a Petersen coloring. It can be easily seen that their girth is 5 and cyclic edge-connectivity is 4. Moreover, each of $M_1,M_2,M_3,M_4$ contains a copy of $W$, so do the graphs $G_{60}, G_{62}, G_{64}$, and $G_{66}$.

  For the induction step, assume that $n\ge 68$. The graph $G_{n-8}$ contains a copy of $W$ by the induction hypothesis. We replace this copy with $M_1$ constructing a graph $G_n$ on $n$ vertices. Again, by Lemma~\ref{lemma:Mka}, the constructed graph $G_n$ is a required counterexample which, moreover, contains a copy of $W$ since $M_1$ contains a copy of $W$.
\end{proof}

 \begin{figure}[ht]
 \centering
  \includegraphics[width=9cm]{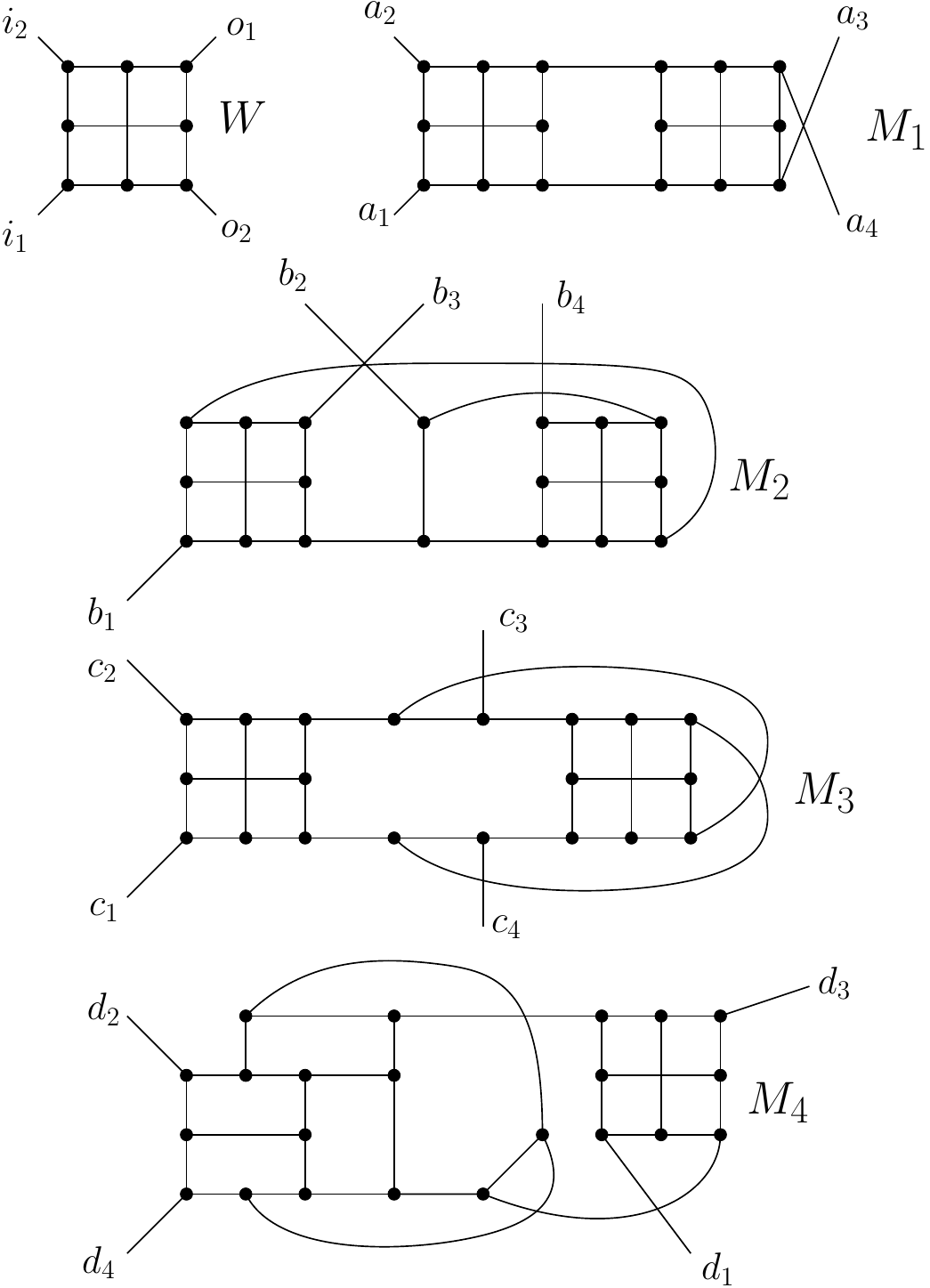}
  \caption{Four 4-poles each having the same set of $P$-colorings as $W$.
} \label{fig:FandPlus2}
 \end{figure}

\section{Further results and open problems} 
\label{sec:openProblems}

In this section, we present further results and pose several open problems.  

\subsection{Smallest possible counterexample}\label{sec:min_counterexample}

 The first very natural open problem is the following. 

\begin{problem}
    What is the smallest bridgeless cubic graph that does not admit a $P$-coloring?
\end{problem}

Following the terminology introduced by Brinkmann et al.~\cite{BGHM13}, we define a \emph{weak snark} as a cyclically $4$-edge-connected cubic graph that is not $3$-edge-colorable, while a \emph{snark} is a weak snark of girth at least $5$. It is well-known that a smallest counterexample to the Petersen Coloring Conjecture must be a weak snark. It was checked in~\cite{BGHM13} that every weak snark on at most $34$ vertices and every snark on at most $36$ vertices admits a $P$-coloring. The first and third authors subsequently completed the order-$36$ case in a joint paper with \v Skoviera~\cite{GMS19} by also generating and testing the weak snarks of girth $4$ on 36 vertices, and showed that Conjecture~\ref{PCC} has no counterexamples on at most $36$ vertices. Consequently, every counterexample to Conjecture~\ref{PCC} has at least $38$ vertices. Recently, Brinkmann and Van Overberghe~\cite{BVO26} developed a new generation algorithm for snarks, which allowed them to generate all weak snarks up to 38 vertices. Using their list of weak snarks on 38 vertices (there are 7\,142\,217\,899 such graphs), we verified that none of them is a counterexample to Conjecture~\ref{PCC}. This computation required approximately 12 CPU years. Together with our counterexamples on 52 vertices, this leads to the following observation.

\begin{observation}\label{obs:order_min_counterex}
The smallest counterexample to Conjecture~\ref{PCC} has at least 40 and at most 52 vertices.
\end{observation}

Note that our counterexamples on 52 vertices\footnote{The two counterexamples can be accessed directly at \url{https://houseofgraphs.org/graphs/57244} and \url{https://houseofgraphs.org/graphs/57278}.} (cf.\ Figures~\ref{fig:52graph} and~\ref{fig:graph_switched}), as well as the other counterexamples on 68 and 112 vertices, can also be obtained from the \textit{House of Graphs}~\cite{CDG23} by searching for the keywords ``Petersen Coloring''.

\subsection{Cyclically 5-edge-connected counterexamples}

In Section~\ref{sec:infinite}, we construct an infinite family of cyclically $4$-edge-connected cubic graphs that do not admit a Petersen coloring. This raises the following natural open problem.

\begin{problem}
    Do there exist cyclically $5$-edge-connected cubic graphs that do not admit a $P$-coloring?
\end{problem}

All counterexamples known so far contain multiple copies of the $4$-pole $W$, and our proof relies heavily on the constraints induced by its coloring properties. This suggests that a cyclically $5$-edge-connected counterexample, if one exists, would likely require a fundamentally different structural approach.

\subsection{Normal colorings}

Recall that a \emph{normal $k$-edge-coloring} of a cubic graph $G$ is a proper edge-coloring $c\colon E(G)\to\{1,\dots,k\}$ such that for every edge $uv\in E(G)$, we have $|\{c(e)\colon e\in \partial_G(u)\cup\partial_G(v)\}|\in \{3,5\}$. Jaeger~\cite{J_5-ec} proved that a cubic graph $G$ admits a $P$-coloring if and only if $G$ has a normal $5$-edge-coloring. The \emph{normal chromatic index} of a graph $G$, denoted by $\chi'_N(G)$, is defined as the minimum integer $k$ such that $G$ admits a normal $k$-edge-coloring.
Consequently, any counterexample $G$ to the Petersen Coloring Conjecture satisfies $\chi'_N(G) > 5$. We computationally verified that both of our counterexamples of order $52$ admit a \textit{strong} normal $6$-edge-coloring (i.e., a normal $6$-edge-coloring where $|\{c(e)\colon e\in \partial_G(u)\cup\partial_G(v)\}|=5$ for every $uv\in E(G)$) and the same holds for the 68-vertex and the two 112-vertex counterexamples. This implies that the normal chromatic index of all of these graphs is exactly~$6$.

Furthermore, it was proved by Mazzuoccolo and Mkrtchyan~\cite{Mazz_Mk2} that every simple cubic graph admits a normal $7$-edge-coloring.\footnote{For the subclass of bridgeless cubic graphs, this is simply a reformulation of Jaeger's 8-flow theorem,} but no examples of bridgeless cubic graphs requiring $7$ colors are currently known. 

The following conjecture remains open; it is attributed to \v{S}\'amal and was also explicitly stated in~\cite{Mazz_Mk}.

\begin{conjecture}\label{conj:normalchromaticindex}
    Let $G$ be a bridgeless cubic graph. Then, $\chi'_N(G)\leq 6$.
\end{conjecture}

Conjecture~\ref{conj:normalchromaticindex} serves as a natural relaxation of Conjecture~\ref{PCC} in the context of normal edge-colorings.

Let $G$ be a graph and $c$ an edge-coloring of $G$. An edge $uv$ of $G$ is called \emph{abnormal} if $|\{c(e)\colon e\in \partial_G(u)\cup\partial_G(v)\}|\notin \{3,5\}$. Let $N_G(c)$ be the set of
abnormal edges of $G$ with respect to $c$. 

In~\cite{MMM} the following question was asked.

\begin{question}[{\cite[Question 3.1]{MMM}}]
    Assume that a bridgeless cubic graph $G$ admits a proper
$5$-edge-coloring $c$, such that $|N_G(c)| \le 2$. Can we prove that $G$ admits a
normal $5$-edge-coloring?
\end{question}

We computationally verified that the two $52$-vertex, one $68$-vertex, and two $112$-vertex counterexamples to the Petersen Coloring Conjecture all admit a proper $5$-edge-coloring with exactly two abnormal edges, so they provide a negative answer to the above question.

\subsection{Family of graphs coloring every bridgeless cubic graph}

If $G$ and $H$ are graphs, a well-known and natural generalization of $P$-coloring of $G$ is the concept of an $H$-coloring of $G$, that is a proper edge-coloring $\varphi\colon E(G)\to E(H)$ such that for every vertex $v\in V(G)$ there is $w_v\in V(H)$ with $\varphi(\partial_G(v))=\partial_H(w_v)$. In such a case we say that $H$ \emph{colors} $G.$

Let $\mathcal{H}$ be an inclusion-wise minimal set of connected cubic graphs such that every connected bridgeless cubic graph $G$ is colored by some $H \in \mathcal{H}$. At this point, one might ask what such a set $\mathcal{H}$ looks like.

Let $S_4$ be the graph obtained as follows: start with a $3$-cycle $xyz$, add a second parallel edge between $y$ and $z$, and add a new vertex $w$ connected to $x$. By combining results of~\cite{KMZ}, \cite{MTZ} and \cite{Mk}, a graph $H$ colors every connected bridgeless cubic graph if and only if $H$ contains $S_4$ as a subgraph, which we call an \emph{$S_4$-graph}. Note that $S_4$ contains a bridge.

As a consequence, a set $\mathcal{H}$ containing an $S_4$-graph $H$ is inclusion-wise minimal if and only if it contains $H$ only. By a result of~\cite{KMZ}, the Petersen graph can only be colored by itself or by an $S_4$-graph. Then, $\mathcal{H}$ has cardinality $1$ if and only if it contains an $S_4$-graph. We believe that the following could be an interesting question: are there inclusion-wise minimal sets $\mathcal{H}$ of finite cardinality greater than $1$? Note that because such a set cannot contain an $S_4$-graph, it must necessarily contain the Petersen graph itself to ensure that the Petersen graph is colored.

On the other hand, if we restrict our scope to bridgeless graphs, let $\mathcal{H}_b$ be an inclusion-wise minimal set of connected \emph{bridgeless} cubic graphs that colors every connected bridgeless cubic graph. By Theorem~1.1 of~\cite{MMSW}, $\mathcal{H}_b$ is infinite. Furthermore, Theorem~1.2 of~\cite{MMSW} implies that $\mathcal{H}_b$ is unique, and its elements are precisely those bridgeless cubic graphs that are colorable only by themselves.

Hence, it would be interesting to investigate whether our counterexamples of order 52 are also colorable only by themselves. This question is of particular relevance, as it is naturally tied to the question of whether they are the smallest bridgeless cubic graphs without a Petersen coloring.

\subsection{Further results on related problems} \label{sec:further_results}

Recall that the \emph{perfect matching index} $\pi(G)$ of a graph $G$ is the smallest number of perfect matchings needed to cover all the edges of $G$. The Berge-Fulkerson Conjecture is equivalent to the assertion that $\pi(G)\le5$ for every bridgeless cubic graph $G$ (see~\cite{Maz}).

%\gi{I would say that it easily follows by the main result in~\cite{MM}} \jan{What exactly do you mean? That because these graphs are not cyclically 5-edge-connected, they cannot be minimal counterexamples to the BF conjecture? But theoretically speaking they might have been non-minimal counterexamples which might be cyclically 4-edge-connected, no? Probably I am overlooking something?} \ed{here I agree with Jan} \gi{I believe we have so many copies of W that we can reduce each component to an order less than 40} \ed{sure, but here you use also the proof not only the statement that Smallest counterexample to BF is cyclically 5ec. I would not say that it EASILY folows from [20]} 

We computationally checked 
that the Berge-Fulkerson Conjecture holds for the two counterexamples of order 52 and the 68-vertex and two 112-vertex counterexamples. The same conclusion follows from the proof of the main result in \cite{MM} by observing that all such examples contain many 4-edge-cuts.

In fact, they even have perfect matching index $4$. In this regard, Hakobyan and Mkrtchyan \cite{HM} made the following conjecture.

\begin{conjecture}[Hakobyan and Mkrtchyan~{\cite[Conjecture 3.2]{HM}}]
    Any cubic graph $G$ with $\pi(G) \le 4$ admits a $P_{12}$-coloring.
\end{conjecture}

Here, $P_{12}$ is the cubic graph obtained by replacing a vertex of the Petersen graph $P$ by a triangle.
We remark that our counterexamples provide a negative answer to this conjecture as well. Indeed, since $P_{12}$ admits a $P$-coloring, our graphs do not admit a $P_{12}$-coloring either. 

Given that such graphs exhibit such a highly specific behavior disproving the Petersen and $P_{12}$-coloring conjectures, while satisfying the Berge-Fulkerson conjecture, it is natural to investigate how they behave with respect to the Cycle Double Cover Conjecture.
Recently, OpenAI posted a preprint~\cite{cdc} in which they proved that every bridgeless graph admits a cycle double cover with eight even subgraphs. It is then natural to ask whether the proposed counterexamples to the Petersen Coloring satisfy the more restrictive $5$-Cycle Double Cover Conjecture. Indeed, we computationally checked that the above-mentioned two $52$-vertex, one $68$-vertex, and two $112$-vertex graphs do admit a cycle double cover with five even subgraphs.

\section*{Acknowledgements}
We would like to thank Gunnar Brinkmann and Steven Van Overberghe for providing us with the complete list of all weak snarks on 38 vertices.

Jan Goedgebeur is supported by Internal Funds of KU Leuven and a grant of the Research Foundation Flanders (FWO) with grant number G0AGX24N. Several of the computations for this work were carried out using the supercomputer infrastructure provided by the VSC (Flemish Supercomputer Center), funded by the Research Foundation Flanders (FWO) and the Flemish Government. 

The role of AI in this work was restricted to independently double-checking the accuracy of the properties of the multipoles under consideration.

% \section*{Appendix: a $52$-vertex counterexample to the Petersen coloring conjecture}

% \[
% \begin{aligned}
% E(G)=\{&
% \{1,4\}, \{1,5\}, \{1,13\}, \{2,7\}, \{2,10\}, \{2,52\},\\
% &
% \{3,4\}, \{3,8\}, \{3,46\}, \{4,9\}, \{5,7\}, \{5,8\},\\
% &
% \{6,8\}, \{6,9\}, \{6,31\}, \{7,9\}, \{10,11\}, \{10,15\},\\
% &
% \{11,12\}, \{11,16\}, \{12,17\}, \{12,20\}, \{13,15\}, \{13,16\},\\
% &
% \{14,16\}, \{14,17\}, \{14,40\}, \{15,17\}, \{18,21\}, \{18,22\},\\
% &
% \{18,30\}, \{19,24\}, \{19,27\}, \{19,52\}, \{20,21\}, \{20,25\},\\
% &
% \{21,26\}, \{22,24\}, \{22,25\}, \{23,25\}, \{23,26\}, \{23,48\},\\
% &
% \{24,26\}, \{27,28\}, \{27,32\}, \{28,29\}, \{28,33\}, \{29,34\},\\
% &
% \{29,37\}, \{30,32\}, \{30,33\}, \{31,33\}, \{31,34\}, \{32,34\},\\
% &
% \{35,38\}, \{35,39\}, \{35,47\}, \{36,41\}, \{36,44\}, \{36,52\},\\
% &
% \{37,38\}, \{37,42\}, \{38,43\}, \{39,41\}, \{39,42\}, \{40,42\},\\
% &
% \{40,43\}, \{41,43\}, \{44,45\}, \{44,49\}, \{45,46\}, \{45,50\},\\
% &
% \{46,51\}, \{47,49\}, \{47,50\}, \{48,50\}, \{48,51\}, \{49,51\}
% \}.
% \end{aligned}
% \]

\end{document}